\documentclass[11pt, english]{smfart}
\usepackage{graphpap}
\usepackage{amssymb}
\newtheorem{theoreme}{Theorem}
\renewcommand{\theequation}{\arabic{equation}}
\newtheorem{proposition}{Proposition}
\newtheorem{lemme}[proposition]{Lemma}
\newtheorem{corollaire}[proposition]{Corollary}
\newtheorem{definition}[proposition]{Definition}
\newtheorem{remarque}[proposition]{Remark}

\numberwithin{equation}{section}
\numberwithin{proposition}{section}

\def\11{{\rm 1~\hspace{-1.4ex}l} }
\def\R{\mathbb R}

\def\O{\mathbb O}

\begin{document}
\title[Supercritical wave equation] {Random data Cauchy theory for supercritical wave equations I: local theory}
\author{Nicolas Burq}
\address{Laboratoire de Math\'ematiques, B\^at. 425, Universit\'e Paris Sud,
  91405  Orsay Cedex, 
France et Institut Universitaire de France}
\email{nicolas.burq@math.u-psud.fr}
\author{Nikolay Tzvetkov}
\address{D\'epartement de Math\'ematiques, Universit\'e Lille I, 59655 Villeneuve d'Ascq Cedex, France}
\email{nikolay.tzvetkov@math.univ-lille1.fr}
\begin{abstract} 
We study the local existence of  strong solutions 
for the cubic nonlinear wave equation with data in $H^s(M)$, $s<1/2$, where $M$
is a three dimensional compact riemannian manifold. This problem is
supercritical and can be shown to be strongly ill-posed (in the Hadamard sense). However, after a suitable randomization,
we are able to construct local strong solution for a large set of initial data
in $H^s(M)$, where $s\geq 1/4$ in the case of a boundary less manifold and $s\geq
8/21$ in the case of a manifold with boundary.
\end{abstract}

\maketitle
%
\section{Introduction}
In the study of the local well-posedness of a nonlinear evolutionary PDE,
one often encounters the presence of a critical threshold for the
well-posedness theory. A typical situation is to have a method showing
well-posedness in Sobolev spaces $H^s$ where s is greater than a critical
index $s_{cr}$. This index is often related to a scale invariance (leading
to solutions concentrating at a point of the space-time) of the considered
equation.  In some cases (but not all), a good local well-posedness theory
is valid all the way down to the scaling regularity. On the other hand, at
least in the context of nonlinear dispersive equations, no reasonable
local well-posedness theory is known for any supercritical equation, i.e.
for data having less regularity than the scaling one. In fact, recently, several methods to show ill-posedness, or high frequency instability, for $s<s_{cr}$ emerged (see the works by Burq, G\'erard and Tzvetkov~\cite{BGT3, BGT2}, Lebeau~\cite{L} and Christ Colliander and Tao~\cite{ChCoTa03}). The goal of this paper is to give a class of equations for which, using probabilistic arguments, one
can still obtain a suitable well-posedness theory below the critical threshold.
Our model will be the cubic nonlinear wave equation posed on a compact manifold.

Let $(M,g)$ be a three dimensional compact smooth  riemannian manifold (without boundary) and let ${\mathbf \Delta}$
be the Laplace-Beltrami operator associated to the smooth metric $g$. For $s\in\R$, we denote by
$H^s(M)$ the classical Sobolev space equipped with the norm $\|u\|_{H^s(M)}=\|(1-{\mathbf \Delta})^{s/2}u\|_{L^2(M)}$.
Consider the following cubic wave equation 
\begin{equation}\label{1}
(\partial_{t}^{2}-{\mathbf \Delta})u+u^3=0,\quad (u,\partial_t u)|_{t=0}=(f_1,f_2)
\end{equation}
with  {\it real valued} initial data 
$
(f_1,f_2)\equiv f\in H^s(M)\times H^{s-1}(M)\equiv {\mathcal H}^{s}(M).
$\par

Using Strichartz estimates for the free evolution (see Section~2) one can show
that for $s>1/2$ the Cauchy problem (\ref{1}) is locally well-posed for data
in ${\mathcal H}^{s}(M)$. This means that for every $f\in {\mathcal H}^{s}(M)$
there exists $T>0$ and a unique solution $u$ of (\ref{1}), in a suitable
class, such that $(u,u_t)\in C([0,T];{\mathcal H}^{s}(M))$, i.e. the solution
$u$ represents a continuous curve in $H^s(M)$ (we call such a solution strong
solution since the classical construction of weak solutions does not yield the
continuity in time). Moreover, we can show that the time
existence $T$ may be chosen the same for all $f$ belonging to a fixed bounded
set $B$ of ${\mathcal H}^{s}(M)$ and the map $f\mapsto (u,u_t)$ is continuous
(and even Lipschitz continuous) from $B$ to $C([0,T];{\mathcal H}^{s}(M))$.

For $s=1/2$ one can still construct local strong solution for $f\in {\mathcal H}^{1/2}(M)$ 
but the dependence of $T$ on $f$ is more complicated and the Sobolev space
${\mathcal H}^{1/2}(M)$ is called critical space for (\ref{1}). 

For $s<1/2$, the argument to construct local solutions by Strichartz estimates
breaks down. Moreover one may show (see \cite{ChCoTa03}, \cite{L} for the case
of constant coefficient metrics or the appendix of this paper for the case of  
non constant coefficient metrics ) that 
if the initial data belong to ${\mathcal H}^s(M)$, $s<1/2$, 
the Cauchy problem (\ref{1}) is ill-posed in a strong sense: there exists
initial data 
$(f_1, f_2) \in \mathcal{H}^s(M)$ such that any {\em reasonable} solution of~\eqref{1}, i.e. satisfying the finite speed of propagation ceases {\em instantaneously}
to be in $\mathcal{H}^s$ for positive times (by finite speed of propagation,
we mean the fact that the value of the solution at $(x_0, t_0)$ depends only
on the values of the initial data on the set of points located at distance
smaller that $|t_0|$: $\{x : d_g(x, x_0)\leq |t_0|\}$).  
However, the functions for which one can prove such a pathological behavior
are highly non generic and a natural question is whether despite this result
one can still prove that the problem~\eqref{1} 
possesses local strong solutions for a ``large class
of functions'' in ${\mathcal H}^s(M)$, $s<1/2$. Our purpose in this paper is
precisely to give a positive answer to this question.
Let us observe that the possibility of such a phenomenon in the context of the nonlinear Schr\"odinger
equation (NLS) is studied in the last section of \cite{Tz2}. However, the situation in the context of NLS 
is much more involved and it would be interesting to decide whether the main result of this paper has an 
appropriate extension in the context of NLS (or other nonlinear PDE's).  

Let us first precise what we mean by ``a large class of initial data in
$\mathcal{H}^s(M)$''. 
Let $e_n\in C^{\infty}(M)$, $n=1,2,\ldots$ be an orthonormal basis of $L^2(M)$ 
constructed from real eigenfunctions of the operator
$-{\mathbf \Delta}$ associated to eigenvalues $ \lambda_n^2 $. Let
$((h_n(\omega),l_n(\omega))_{n=1}^{\infty}$ 
be a sequence of independent, $0$ mean value, real random variables on a
probability space $(\Omega,{\mathcal A},p)$ 
such that
\begin{equation}\label{norm}
\exists\, C>0\,:\, \forall\,n\geq 1,\,\,
\int_{\Omega}(|h_{n}(\omega)|^{4}+|l_{n}(\omega)|^{4})dp(\omega)<C\,.
\end{equation}
For  $f=(f_1,f_2)\in {\mathcal H}^s(M)$ given by 
\begin{equation*}
f_1(x)=\sum_{n=1}^{\infty}\alpha_n e_{n}(x),\quad
f_2(x)=\sum_{n=1}^{\infty}\beta_n e_{n}(x),\qquad \alpha_n, \beta_n \in \mathbb{R},
\end{equation*} 
we consider the map 
\begin{equation}\label{rand}
\omega\longmapsto f^{\omega}=(f_{1}^{\omega},f_2^{\omega})
\end{equation}
from $(\Omega,{\mathcal A})$  to
${\mathcal H}^s(M)$ equipped with the Borel sigma algebra, defined by
\begin{equation}\label{rand2}
f_1^\omega(x)=\sum_{n=1}^{\infty} h_{n}(\omega)\alpha_n e_{n}(x),\quad
f_2^\omega(x)=\sum_{n=1}^{\infty} l_{n}(\omega)\beta_n e_{n}(x)\,.
\end{equation}
Using (\ref{norm})
one can check that the map $\omega\mapsto f^{\omega}$ is measurable and  
$f^{\omega}\in L^{2}(\Omega;{\mathcal H}^s(M))$. 
Thus it defines a ${\mathcal H}^s(M)$ valued random variable. 
A simple computation (see Appendix~2) shows that if $h_n,l_n$ are identically distributed and different
from zero, or more generally if there exists $c>0$ such that the distributions $h_n, l_n$ satisfy
$$\limsup_{n\rightarrow + \infty}p(\{ |h_n| +|l_n| \leq c\})<1,$$
then, if $f$ does not belong to ${\mathcal H}^{s+\varepsilon}(M)$, for almost
all $\omega$, $f^{\omega}$ 
does not belongs to ${\mathcal H}^{s+\varepsilon}(M)$. Thus the randomization $\omega\mapsto f^{\omega}$ does not
give a regularization in the scale of the Sobolev spaces (but we shall
crucially exploit $L^p$ regularizations of this randomization).
Our main result reads as follows.
\begin{theoreme}\label{thm1} 
Assume that $\partial M = \emptyset$. Let us fix $s\geq\frac 1 4$ and $f=(f_1,f_2)\in {\mathcal H}^{s}(M)$.
Let $f^{\omega}\in L^{2}(\Omega;{\mathcal H}^s(M))$ be defined by the
randomization (\ref{rand}), (\ref{rand2}). Then there exists $\sigma\geq\frac 1 2$ such that for almost 
all $\omega\in \Omega$ there exist $T_\omega>0$ and a unique solution
to~\eqref{1} with initial data $f^{\omega}$ in a space continuously embedded in 
$$ 
X_\omega = \Big(\cos(t\sqrt{- {\mathbf \Delta}}) f_1^\omega+
\frac{\sin ( t \sqrt{- {\mathbf \Delta}})f_2^\omega}{ \sqrt{ - {\mathbf \Delta}}}\Big) + C([-T_\omega, T_\omega]; H^{\sigma}(M)).
$$
More precisely, there exist $C>0, \delta\geq 0$ ($\delta>0$ if $s>1/4$) and
for every 
$0<T\leq 1$,  an event $\Omega_{T}$ such that
\begin{equation}
\label{eq.prop}
p(\Omega_{T})\geq 1-CT^{1+ \delta}
\end{equation} and such that for every
$\omega\in\Omega_T$ there exists a unique solution $u$ of (\ref{1})
with data $f^{\omega}$ in a space continuously embedded in
$C([-T,T];H^s(M))$.
\\
Moreover, if $s>1/4$ and $h_n,l_n$ are standard real Gaussian or Bernoulli variables~\eqref{eq.prop} can be improved to
\begin{equation}
\label{eq.propbis}
p(\Omega_{T})\geq 1-Ce^{-c/T^\delta},\quad c,\, \delta>0.
\end{equation}
\end{theoreme}
Let us make several remarks.
\begin{remarque} 
{\rm
The result in Theorem~\ref{thm1} shows that in some sense the problem~\eqref{1} is 
well-posed for almost all initial data in ${\mathcal H}^{\frac 1 4 }(M)$, 
exhibiting a gain of $1/4$ derivatives with respect to the critical index $1/2$.
}
\end{remarque}
\begin{remarque}
{\rm
For any $f\in \mathcal{H}^s(M), 1/4 \leq s$, the map
$$ \omega \in \Omega \mapsto f^\omega \in \mathcal{H}^s(M)$$
endows naturally $\mathcal{H}^s(M)$ with a probability measure $\mu_f$.
It is straightforward to check that the solutions given by Theorem~\ref{thm1} satisfy the finite speed of propagation.  As a consequence, Theorem~\ref{thm1} implies that the set of initial data exhibiting the same kind of pathological behavior as the ones we  constructed in the appendix have measure $0$ for any measure $\mu_f$.
}
\end{remarque}
\begin{remarque} 
{\rm
Combining the ideas developed in this paper (and in particular~\eqref{eq.propbis}) with a global control on the flow given by invariant measures (which itself is related to the Hamiltonian nature of our equation, see e.g. our previous work~\cite{BuTz07} or \cite{Bo1,Bo2,Tz1,Tz2}) leads to {\em global} well posedness results for a class of {\em super critical wave equations} (see our forthcoming paper~\cite{BuTz07-2}). On the other hand, it would be interesting to decide whether the local in time result of Theorem~\ref{thm1} may be successfully combined with other global controls on the flow  such as conservation laws.
}
\end{remarque}
\begin{remarque}
{\rm
The method of proof consists in using the fact that though the initial data
have low regularity, their $L^p$ properties are (almost surely) much better than expected, allowing the
use of a fixed point method after having singled out the linear evolution. 
Let us note that such $L^p$ regularization phenomena are well-known in the
context of Fourier series since the work of Paley-Zygmund~\cite{PZ}.
Similar phenomena were recently studied by Ayache and the second author in the context of sums of type
(\ref{rand2}) in \cite{AT}.
}
\end{remarque}
\begin{remarque}
{\rm
In the improved time existence statement of Theorem~\ref{thm1} one may replace the assumption to deal
with Gaussian or Bernoulli's variables by the assumption~(\ref{property}) below.
}
\end{remarque}
\begin{remarque}
{\rm
For the sake of conciseness, we chose to focus on the cubic semi linear wave equation. However, the strategy presented here applies to arbitrary non linearities and allows to go beyond the usual critical threshold.
}
\end{remarque} 
Finally let us notice that our results extend to Dirichlet or Neumann boundary conditions. 
In this case, the deterministic Cauchy theory is much less well known.
\begin{definition} 
Let $(e_n)_{n=1}^{\infty}$ be the $L^2$-normalized basis consisting in eigenfunctions of the Laplace operator with Dirichlet (resp Neumann) boundary conditions, associated to eigenvalues $\lambda_n^2$.
The space $H^s_D(M)$ (resp $H^s_N(M)$) is the space of functions $f$ such there exists a sequence $(\alpha_n)_{n=1}^{\infty}$ such that 
\begin{equation}\label{mai}
f(x)= \sum_n \alpha _n e_n(x)
\end{equation}
with 
$$
\sum_{n}(1+\lambda_n)^{2s}|\alpha_n|^{2}<\infty\,.
$$
We shall denote by $\mathcal{H}^s_D(M)= H^s_D(M) \times H^{s-1}_D (M)$ (resp $\mathcal{H}^s_N= H^s_N(M) \times H^{s-1}_N (M)$).
\end{definition}
\begin{remarque}
{\rm
The space $H^s_D(M)$ coincides with the usual Sobolev space of order $s$ if $-1/2<s<1/2$ 
whereas $H^s_N(M)$ coincide with the usual Sobolev space of order $s$ if $-3/2<s<3/2$.
}
\end{remarque}
Consider now the wave equation
\begin{equation}\label{1bis}
(\partial_{t}^{2}-{\mathbf \Delta})u+u^3=0,\quad (u,\partial_t u)|_{t=0}=(f_1,f_2)
\end{equation}
with  {\it real valued} initial data $
(f_1,f_2)\equiv f\in H^s(M)\times H^{s-1}(M)\equiv {\mathcal H}^{s}(M),
$ and Dirichlet ($ u \mid_{\mathbb{R}_t \times \partial M} =0$) or Neumann
($ \frac{ \partial u } { \partial n} \mid_{\mathbb{R}_t \times \partial M} =0$) 
boundary conditions.
 
 From the results by Lebeau, Planchon and the first author~\cite{BuLePl06} (see also section~\ref{sec.bord}), one can show that the Cauchy problem is well posed in $\mathcal{H}_{D,N}^{2/3}(M)$. Here we show that almost surely, this result can be improved to $s=8/21<1/2$). 
\begin{theoreme}\label{thm2} 
Assume that $\partial M \neq \emptyset$, and that the random variables have sixth moments uniformly bounded
\begin{equation}\label{normbis}
\exists\, C>0\,:\, \forall\,n\geq 1,\,\,
\int_{\Omega}(|h_{n}(\omega)|^{6}+|l_{n}(\omega)|^{6})dp(\omega)<C\,.
\end{equation}
 Let us fix $s\geq\frac 8 {21}$ and $f=(f_1,f_2)\in {\mathcal H}^{s}_{D,N}(M)$.
Let $f^{\omega}\in L^{2}(\Omega;{\mathcal H}^s_{D,N}(M))$ be defined by the
randomization (\ref{rand}), (\ref{rand2}). Then there exists $\sigma\geq \frac 2 3$ such that for almost 
all $\omega\in \Omega$ there exist $T_\omega>0$ and a unique solution
to~\eqref{1bis} with initial data $f^{\omega}$ in a space continuously embedded in 
$$ 
X_\omega = \Big(\cos(t\sqrt{- {\mathbf \Delta_{D,N}}}) f_1^\omega+
\frac{\sin ( t \sqrt{- {\mathbf \Delta_{D,N}}})f_2^\omega}{ \sqrt{ - {\mathbf\Delta_{D,N}}}}\Big) + 
C([-T_\omega, T_\omega]; H^{\sigma}_{D,N}(M)).
$$
More precisely, there exists $C>0, \delta> 0$  and  for every $0<T\leq 1$,  an event $\Omega_{T}$ such that
\begin{equation}\label{prop1}
p(\Omega_{T})\geq 1-CT^{1+ \delta}
\end{equation} and such that for every
$\omega\in\Omega_{T}$ there exists a unique solution $u$ of (\ref{1})
with data $f^{\omega}$ in a space continuously embedded in
$C([0,T];H^s(M))$.

Moreover, if $h_n,l_n$ are standard real Gaussian or Bernoulli variables one
can improve~\eqref{prop1} to 
\begin{equation}\label{prop2}
p(\Omega_{T})\geq 1-Ce^{-c/T^\delta}, \quad c, \delta >0\,.
\end{equation}
\end{theoreme}
The paper is organized as follows: in section~\ref{sec.stri} we recall Strichartz estimates for wave equations and Sogge's estimates for $L^p$ norms of spectral projectors. In Section~\ref{sec.dev} we prove a large deviation bound. In Section~\ref{sec.ave} we prove in some sense that ``randomization beats deterministic Strichartz estimates'' in terms of $L^p$ estimates. In Section~\ref{sec.fix} we perform a classical fixed point argument in a suitable space to prove Theorem~\ref{thm1}. Finally, in Section~\ref{sec.bord} we indicate how the previous argument have to be adapted in the case of a boundary value problem. In all the proof, we shall focus on positive times, the case of negative times being similar due to time reversibility.

{\bf Acknowledgments:} We thank H.~Queff\'elec for providing us the reference~\cite{PZ}.

\section{Strichartz and Sogge estimates}\label{sec.stri}
We shall assume in this section that the boundary of $M$ is empty and collect the Strichartz estimates for the free evolution and the Sogge estimates for
the eigenfunctions $e_{n}(x)$. These sets of estimates are actually in the same family, their proofs
being a combination of the Fourier integral operator approximation for hyperbolic problems and the
$TT^{\star}$ duality argument. 
Let us note that due to the finite speed of propagation for solutions to the wave equation, 
Strichartz estimates on a compact manifold are equivalent to some variable coefficient 
Strichartz estimates on $\R^3$.
We refer to  Kapitanskii~\cite{Ka} for the proof of the Strichartz estimate
we state bellow and to Sogge~\cite{So} for the bound on $e_{n}$ stated bellow.

The purpose of the next definition is to define the Strichartz spaces used for solving the problem (\ref{1}).
\begin{definition}
For $0\leq s < 1$, a couple of real numbers $(p,q), \frac 2 s\leq p\leq + \infty$ is $s$-admissible if
$$\frac{1}{p}+\frac{3}{q}=\frac{3}{2}-s.
$$
For $T>0$, $0\leq s < 1$, we define the spaces 
\begin{equation}\label{eq.espstri}
 X^s_T= C^0 ([0, T]; H^s( M)) \bigcap_{(p,q) \text{ $s$- admissible}} L^p((0, T) ; L^q(M))
\end{equation}
and its dual space
\begin{equation}\label{eq.spdual}
 Y^s_T= L^1 ([0, T]; H^{-s}(M)) +_{(p,q) \text{ $s$- admissible}} L^{p'}((0, T) ; L^{q'}(M))
\end{equation}
$(p',q')$ being the conjugate couple of $(p,q)$, equipped with their natural norms (notice that to define these spaces, we keep only the extremal couples corresponding to $p= 2/s$ and $p= + \infty$ respectively).
\end{definition} 
We next state the Strichartz inequality for the wave equation, posed on a three dimensional smooth 
compact (without boundary) riemannian manifold.
\begin{proposition}\label{str}
Let $(p,q)$ be an $s$-admissible couple. Then there exists $C>0$ such that for every $T\in]0,1]$, 
every $f\in H^{s}(M)$ one has
\begin{equation}\label{eq.homog}
\|e^{\pm it\sqrt{- {\mathbf \Delta}}}(f)\|_{L^{p}([-T,T];L^{q}(M))}\leq 
C\|f\|_{H^{s}(M)}\,.
\end{equation}
\end{proposition}
Let us now state a corollary of Proposition~\ref{str}.
\begin{corollaire}\label{cor.stri}
For every $0< s <1 $, every $s$-admissible couple $(p,q)$, there exists $C>0$ such that
for every $T\in]0,1]$, every $f\in \mathcal{H}^{s}(M), g\in Y^{1-s}_T$ one has
\begin{equation}\label{eq.homogbis}
\|\cos({t\sqrt{- {\mathbf \Delta}}})(f_1)\|_{X^s_T}+\|\frac{\sin({t\sqrt{- {\mathbf \Delta}}})}{ \sqrt{- \mathbf{\Delta}}}(f_2)\|_{X^s_T}\leq 
C\|f\|_{\mathcal{H}^{s}(M)},
\end{equation}
\begin{equation}\label{eq.inhomog}
\|\int_0^t \frac {\sin((t-\tau)
\sqrt{ - {\mathbf \Delta}})}{\sqrt{-{\mathbf \Delta}}} (g)(\tau) d\tau\|_{X^s_T}\leq C \|g\|_{Y^{1-s}_T}
\end{equation} 
\end{corollaire}
The proof of Corollary~\ref{cor.stri} can essentially be found in \cite{BuTz07}. Notice that here we have to modify slightly the argument to take care of the $0$ eigenvalue of the Laplace operator.
We next state the Sogge estimate which will be involved in our analysis. 
\begin{proposition}\label{chris}
There exists $C>0$ such that for every $n\geq 1$,
$$
\|e_n\|_{L^{4}(M)}\leq C(1+\lambda_{n}^2)^{\frac{1}{8}}\,.
$$
\end{proposition}
Let us note that the $L^{4}(M)$ norm can be replaced by other $L^{p}(M)$, $2\leq p\leq \infty$ 
norms by modifying appropriately the power of $1+\lambda_n^2$ according to an interpolation 
with the trivial $L^2$ bound or the $L^{\infty}$ Weyl bound.
We also note that the estimate of Proposition~\ref{chris} has a natural extension to other dimensions, 
the index $4$ being replaced by $2(d+1)/(d-1)$.
Finally, the estimate for $e_n$ given by  Proposition~\ref{chris} also holds, with the appropriate statement, 
for the spectral projection on $\sqrt{-{\mathbf \Delta}}\in [\lambda,\lambda+1]$.
\section{A large deviation bound}\label{sec.dev}
The purpose of this section is to prove the following statement.
\begin{lemme}\label{lem1}
Let $(l_n(\omega))_{n=1}^{\infty}$ 
be a sequence of real, independent random variables with associated sequence of
distributions $(\mu_{n})_{n=1}^{\infty}$.
Assume that $\mu_{n}$ satisfy the property
\begin{equation}\label{property}
\exists\, c>0\,:\, \forall\,\gamma\in\R,\,\forall\, n\geq 1,\,
\Big|\int_{-\infty}^{\infty}e^{\gamma x}d\mu_{n}(x)\Big|\leq e^{c\gamma^{2}}\,.
\end{equation}
Then there exists $\alpha>0$ such that 
for every $\lambda >0$, every sequence $(c_n)_{n=1}^{\infty}\in l^2$ of real numbers,
\begin{equation}\label{khin}
p\Big(\omega\,:\,\big|\sum_{n=1}^{\infty} c_n l_n(\omega) \big|>\lambda\Big)\leq
2 e^{-\frac{\alpha\lambda^2}{\sum_{n}c_n^2}}
\, .
\end{equation}
As a consequence there exists $C>0$ such that for every $p\geq 2$, every $(c_n)_{n=1}^{\infty}\in l^2$,
\begin{equation}\label{khinbis}
\big\|\sum_{n=1}^{\infty} c_n l_n(\omega) \big\|_{L^{p}(\Omega)}
\leq C\sqrt{p}\big(\sum_{n=1}^{\infty}c_n^2\big)^{1/2}
.
\end{equation}
\end{lemme}
\begin{remarque}
{\rm
Let us notice that (\ref{property}) is readily satisfied if $(l_n(\omega))_{n=1}^{\infty}$ are standard 
real Gaussian or standard Bernoulli variables. Indeed in the case of Gaussian
$$
\int_{-\infty}^{\infty}e^{\gamma x}d\mu_{n}(x)=
\int_{-\infty}^{\infty}e^{\gamma x}\, e^{-x^2/2}\frac{dx}{\sqrt{2\pi}}
=e^{\gamma^2/2}\,.
$$
In the case of  Bernoulli variables (or more generally any random variables having compactly supported distribution) one can obtain that (\ref{property}) 
is satisfied by invoking the inequality
$$
\frac{e^{\gamma }+e^{-\gamma }}{2}\leq e^{\gamma^2/2},\quad \forall\, \gamma\in\R.
$$
}
\end{remarque}
\begin{proof}[Proof of Lemma~\ref{lem1}.]
We give an argument similar to the proof of \cite[Lemma~4.2]{Tz3}.
In the case of Gaussian we can see Lemma~\ref{lem1} as a very particular case of a $L^p$ smoothing 
properties of the Hartree-Foch heat flow (see e.g. \cite[Section~3]{Tz3} for more details on this issue).
For $t>0$ to be determined later, using the independence and (\ref{property}), we obtain 
\begin{multline*}
\int_{\Omega}\, e^{t\sum_{n\geq 1}c_n l_n(\omega)}dp(\omega)
= \prod_{n\geq 1}\int_{\Omega}e^{t c_n l_n(\omega)}dp(\omega)
\\
=  \prod_{n\geq 1}\int_{-\infty}^{\infty}e^{tc_n x}\, d\mu_{n}(x)
\leq 
\prod_{n\geq 1}e^{c(t c_n)^2}= e^{(ct^2)\sum_{n}c_n^2}\, .
\end{multline*}
Therefore
$$
e^{(ct^2)\sum_{n}c_n^2}\geq e^{t\lambda}\,\,\,  p\,(\omega\,:\,\sum_{n\geq 1} c_n l_n(\omega)>\lambda)
$$
or equivalently,
$$
p\,(\omega\,:\,\sum_{n\geq 1} c_n l_n(\omega)>\lambda)\leq e^{(ct^2)\sum_{n}c_n^2}\,\,\,
e^{-t\lambda}\, .
$$
We choose $t$ as
$$
t\equiv \frac{\lambda}{2c \sum_{n}c_n^2}\,.
$$
Hence
$$
p\,(\omega\,:\,\sum_{n\geq 1} c_n l_n(\omega)>\lambda)\leq
e^{-\frac{\lambda^2}{4c\sum_{n}c_n^2}}\, .
$$
In the same way (replacing $c_n$ by $-c_n$), we can show that
$$
p\,(\omega\,:\,\sum_{n\geq 1} c_n l_n(\omega)<-\lambda)\leq
e^{-\frac{\lambda^2}{4c\sum_{n}c_n^2}}
$$
which completes the proof of~\eqref{khin}. To deduce~\eqref{khinbis}, we  write
\begin{multline*}
\|\sum_{n=1}^{\infty} c_n l_n(\omega)\|^p_{L^p(\Omega)}= 
{p} \int_0^{+\infty} p(\omega\,:\, |\sum_{n=1}^{\infty} c_n l_n(\omega)|>\lambda ) \lambda^{p-1} d \lambda\\
\leq Cp \int_0^{+\infty} \lambda^{p-1} e^{-\frac{c \lambda^2}{\sum_n c_n^2}} d\lambda \leq Cp  (C\sum_n c_n^2)^{\frac p 2}\int_0^{+\infty} \lambda^{p-1} e^{-\frac{\lambda^2} 2} d\lambda \leq C  (Cp \sum_n c_n^2)^{\frac p 2}
\end{multline*}
which completes the proof of Lemma~\ref{lem1}.
\end{proof}
\section{Averaging effects}\label{sec.ave}
In this section, we exploit the randomization to get two $L^4$ estimates for the free evolution.
These estimates play a central role in the proof of Theorem~\ref{thm1}.
\begin{proposition}\label{prop.aver}
Let $s\geq 1/ 4$ and $0<T\leq 1$.
Under the assumptions of Theorem~\ref{thm1}, 
for $f=(f_1,f_2)\in {\mathcal H}^{s}(M)$, we consider the free evolution with data $f^{\omega}$, given by 
$$
u_{f}^\omega(t,x) = 
\cos(t\sqrt{- {\mathbf \Delta}}) f_1^\omega+\frac{\sin ( t \sqrt{- {\mathbf \Delta}})}{ \sqrt{ - {\mathbf \Delta}}}f_2^\omega\,.
$$
Then there exists $C>0$ such that for every $f\in {\mathcal H}^{s}(M)$,
\begin{equation}\label{aver}
\|(-{\mathbf \Delta} +1)^{\frac s 2 - \frac 1 8}u_f^\omega\|_{L^4(\Omega_\omega\times [0,T]_t\times M_x)} 
\leq C T^{1/4}\|f\|_{{\mathcal H}^{s}(M)}\,.
\end{equation}
In particular, thanks to the Bienaym\'e-Tchebichev inequality, if we set
$$
E_{\lambda,T,f}= \Big\{ \omega \in \Omega\,:\,
\|(-{\mathbf \Delta} +1)^{\frac s 2 - \frac 1 8}u_f^\omega\|_{L^4([0,T]_t\times M_x)} \geq \lambda\Big\}
$$
then there exists $C>0$ such that for every $\lambda>0$, every $f\in {\mathcal H}^{s}(M)$,
$$ 
p(E_{\lambda,T,f})\leq CT \lambda ^{-4}\,
\|f\|^4_{{\mathcal H}^{s}(M)}\,.
$$
\end{proposition}
\begin{proof}
By expanding the sines and cosines functions as sums of exponentials, we obtain that
we may only consider the contribution of  
$(-{\mathbf \Delta} +1)^{\frac s 2 - \frac 1 8}e^{it\sqrt{ - {\mathbf \Delta}}} f_1^\omega$, 
the other contributions being dealt with similarly (again the zero frequency
should be treated separately).
Suppose that
\begin{equation*}
f_1(x)= \sum_n \alpha _n e_n(x)\in H^{s}(M).
\end{equation*}
If we set
$\widetilde \alpha_n = (\lambda_n^2 +1) ^{\frac s 2 - \frac 1 8} \alpha_n$ 
then 
$$ \|f_1\|^2_{H^s(M)}=\sum_{n} 
|\widetilde \alpha_n| ^2 (1+\lambda_n^2)^{\frac 1 4}\,.
$$
We shall use the following result.
\begin{lemme}\label{lem.k} 
Assume that $(h_n)_{n=1}^{\infty}$ is a sequence of independent, $0$-mean value, complex random variables  satisfying for some 
$k\in \mathbb{N}^*$
$$ 
\exists\, C>0,\, 
\forall\, n \geq 1,\, \int _\Omega  |h_n(\omega)|^{2k} dp (\omega) \leq C
$$
then 
\begin{equation} \label{eq.lp}
\begin{gathered}
\forall\, 2\leq p \leq 2k,
\quad \exists\, C>0,\quad \forall\, (c_n)_{n \in \mathbb{N}^*}\in l^2(\mathbb{N}^*, \mathbb{C}),\\
 \|\sum_n c_nh_n\|_{L^p (\Omega)} \leq C(\sum_n |c_n|^2)^{1/2}\,.
\end{gathered}
\end{equation}
\end{lemme}
\begin{proof} Using H\"older inequality, it suffices to prove the estimate for $p=2k$.
We have
\begin{multline*}
\int_\Omega |\sum_n c_n h_n(\omega)|^{2k}
=\sum_{n_1,\cdots, n_{2k}} \int_\Omega c_{n_1}\times\cdots\times c_{n_k}\overline {c_{n_{k+1}}\times\cdots\times c_{n_{2k}}} \\
 h_{n_1}(\omega)\times\cdots\times h_{n_k}(\omega)\overline {h_{n_{k+1}}(\omega)\times\cdots\times h_{n_{2k}}(\omega)} dp (\omega).
\end{multline*}
Using the independence and the fact that the random variables have $0$ mean, we obtain that for the contribution of $n_1, \cdots, n_{2k}$ not to vanish, each index have to be appear at least twice. As a consequence (using that the $2k$-th moment of the random variables are uniformly bounded)
\begin{equation*}\begin{aligned}
\int_\Omega |\sum_n c_n h_n(\omega)|^{2k}&\leq \sum_{n_1, \cdots, n_k}\int_\Omega |c_{n_1}  \cdots c_{n_k}|^2|h_{n_1}(\omega)\cdots h_{n_{k}}( \omega)|^2 dp (\omega)\\
& \leq C   (\sum_{n} |c_{n} |^2)^k.
\end{aligned}
\end{equation*}
\end{proof}
Let us come back to the proof of Proposition~\ref{prop.aver}. Using Lemma~\ref{lem.k}, we obtain
\begin{equation*}
\begin{aligned}
\|(-{\mathbf \Delta} +1)^{\frac s 2- \frac 1 8} 
e^{it \sqrt{ - {\mathbf \Delta}}} f_1^\omega\|_{L^4(\Omega_\omega\times [0,T]_t\times M_x)}&\leq C \Bigl\|
\Bigl(\sum_{n}|\widetilde \alpha _{n}
e_n(x)|^2\Bigr)^{1/2}\Bigr\|_{L^4((0,T)\times M)}
\\
&= C \Bigl\|
\sum_{n}|\widetilde \alpha _{n} e_n(x)|^2\Bigr\|_{L^2((0,T)\times M)}^{1/2}\\
&\leq  C
\Bigl(\sum_{n} \Bigl\||\widetilde \alpha _{n} e_n(x)|^2 \Bigr\|_{L^2((0,T)\times M)}\Bigr)^{1/2}\\
&\leq  CT^{1/4}\Bigl(
\sum_{n} \|\widetilde \alpha _{n} e_n(x)\|^2_{L^4(M)}\Bigr)^{1/2}
\end{aligned}
\end{equation*}
which, according to Proposition~\ref{chris} implies Proposition~\ref{prop.aver}.
\end{proof}
\begin{remarque} 
{\rm
For $1<p<+\infty$, the norm  $ \|f\|_{W^{s,p}(M)}$ and 
$\|(- {\mathbf \Delta} +1)^{\frac s 2}\|_{L^p(M)}$ are equivalent. Indeed, for 
$s\in 2 \mathbb{N}$, this is a consequence of the $L^p$ elliptic regularity
theorem and for general $s$ it follows by interpolation and duality.
}
\end{remarque}
As a consequence of Lemma~\ref{lem1}, under assumption (\ref{property}), 
we can improve the averaging effect estimate as follows.
\begin{proposition}\label{prop.aver.pak}
Under the assumption of Proposition~\ref{prop.aver}, if we suppose that in addition the randomization obeys
the condition (\ref{property}), then for $p\geq 4$,
\begin{equation}\label{chaos}
\|(-{\mathbf \Delta} +1)^{\frac s 2 - \frac 1 8}u_f^\omega\|_{L^p(\Omega;L^4( [0,1]_t\times M_x))} 
\leq C p^{\frac{1}{2}}\|f\|_{{\mathcal H}^{s}(M)}\,.
\end{equation}
As a consequence, if we set
$$
E_{\lambda,f}= \Big\{ \omega \in \Omega\,:\,
\|(-{\mathbf \Delta} +1)^{\frac s 2 - \frac 1 8}u_f^\omega\|_{L^4([0,1]_t\times M_x)} \geq \lambda\Big\}
$$
then there exist $C>0$ and $c>0$ such that for every $\lambda>0$, every $f\in {\mathcal H}^{s}(M)$,
\begin{equation}\label{iyud} 
p(E_{\lambda,f})\leq C\,e^{-c\lambda ^{2}/\|f\|^2_{{\mathcal H}^{s}(M)}}\,\,.
\end{equation}
\end{proposition}
\begin{proof}
As in the proof of Proposition~\ref{prop.aver}, we may only consider the contribution of  
$(-{\mathbf \Delta} +1)^{\frac s 2 - \frac 1 8}e^{it\sqrt{ - {\mathbf\Delta}}} f_1^\omega$. Writing
$
f_1= \sum_n \alpha _n e_n
$
and if we set
$\widetilde \alpha_n = (\lambda_n^2 +1) ^{\frac s 2 - \frac 1 8} \alpha_n$ 
then 
$$ \|f_1\|^2_{H^s(M)}=\sum_{n} 
|\widetilde \alpha_n| ^2 (1+\lambda_n^2)^{\frac 1 4}\,.
$$
Set
$$
v_{f_1}^\omega(t,x)\equiv(-{\mathbf \Delta} +1)^{\frac s 2 - \frac 1 8}e^{it\sqrt{ - {\mathbf \Delta}}} f_1^\omega\,.
$$
Thus the issue is to show that
\begin{equation*}
\| v_{f_1}^\omega(t,x)\|_{L^{p}(\Omega;L^4( [0,1]_t\times M_x))} 
\leq C \sqrt{p}\|f_1\|_{H^s(M)}\,.
\end{equation*}
By the Minkowski inequality, for $p\geq 4$,
\begin{equation*}
\| v_{f_1}^\omega(t,x)\|_{L^{p}(\Omega;L^4( [0,1]_t\times M_x))} 
\leq \| v_{f_1}^\omega(t,x)\|_{L^4( [0,1]_t\times M_x;L^{p}(\Omega))}\,. 
\end{equation*}
Thanks to Lemma~\ref{lem1},
$$
\|v_{f_1}^\omega(t,x)\|_{L^{p}(\Omega)}=
\|\sum_n\widetilde{\alpha}_ne^{it\lambda_n} e_n(x) h_n(\omega)\|_{L^{p}(\Omega)}\leq
C\sqrt{p}\Bigl(\sum_n |\widetilde{\alpha_n}e_n(x)|^2\Bigr)^{1/2}.
$$
Therefore, we get, thanks to Proposition~\ref{chris},
\begin{eqnarray*}
\| v_{f_1}^\omega(t,x)\|_{L^{p}(\Omega;L^4( [0,1]_t\times M_x))} 
& \leq &
C\sqrt{p}\Bigl\|\Bigl(\sum_n |\widetilde{\alpha_n}e_n(x)|^2\Bigr)^{1/2}\Bigr\|_{L^{4}([0,1]_t\times M_x)}
\\
& \leq &
C\sqrt{p}\Bigl(\Bigl\|\sum_n |\widetilde{\alpha_n}e_n(x)|^2\Bigr\|_{L^{2}(M_x)}\Bigr)^{1/2}
\\
& \leq &C\sqrt{p}\Bigl(\sum_n\Bigl\| |\widetilde{\alpha_n}e_n(x)|^2\Bigr\|_{L^{2}(M_x)}\Bigr)^{1/2}\\
&\leq &C\sqrt{p}\Bigl(\sum_n|\widetilde{\alpha_n}|^2\Big\| e_n(x)\Bigr\|^2_{L^{4}(M_x)}\Bigr)^{1/2}\\
&\leq &C\sqrt{p}\Bigl(\sum_n|\widetilde{\alpha_n}|^2(1+ \lambda_n^2)^{1/4}\Bigr)^{1/2}
\end{eqnarray*}
which completes the proof of (\ref{chaos}).
Let us now turn to the proof of (\ref{iyud}).
Thanks to the Bienaym\'e-Tchebichev inequality, there exists $\alpha>0$ such that for every $p\geq 4$, every
$f\in {\mathcal H}^{s}(M)$, 
\begin{equation}\label{optim}
p(E_{\lambda,f})\leq \lambda^{-p}\Big(\alpha\sqrt{p}\|f\|_{{\mathcal H}^s(M)}\Big)^{p}\,.
\end{equation}
Inequality (\ref{iyud}) easily holds, if $\lambda$ is such that
\begin{equation}\label{4}
\frac{\lambda }{\|f\|_{{\mathcal H}^s(M)}}\leq 2\alpha e\,.
\end{equation}
If (\ref{4}) does not hold, we set
$$
p\equiv\Big[\frac{\lambda }{\alpha\|f\|_{{\mathcal H}^s(M)}e}\Big]^{2}\,\, (\geq 4).
$$
With this choice of $p$, we come back to (\ref{optim}) which yields (\ref{iyud}).
This completes the proof of Proposition~\ref{prop.aver.pak}. 
\end{proof}
\section{The fixed point}\label{sec.fix} 
If we wish to solve 
\begin{equation*}
(\partial_{t}^{2}-{\mathbf \Delta})u+u^3=0,\quad (u,\partial_t u)|_{t=0}=(f^{\omega}_1,f^{\omega}_2)=f^{\omega}
\end{equation*}
by writing $u = u_{f}^{\omega}+ v$, where $u_{f}^{\omega}$ denotes the free evolution associated to 
$f^{\omega}$, we obtain that $v$ solves
\begin{equation}\label{2}
(\partial_{t}^{2}-{\mathbf \Delta})v=-(u_f^{\omega}+v)^3, \quad (v,\partial_t v)|_{t=0}=(0,0).
\end{equation}
Write this equation as 
$$ v(t, \cdot) = -\int_0^t 
\frac{\sin((t-\tau) \sqrt{ - {\mathbf \Delta}})} { \sqrt{ - {\mathbf \Delta}}}((u_f^\omega+v)^3) (\tau, \cdot) d\tau.
$$
Define the map
$$
K_{f}^{\omega} \,:\,v
\longmapsto 
-\int_0^t 
\frac{\sin((t-\tau) \sqrt{ - {\mathbf \Delta}})} { \sqrt{ - {\mathbf \Delta}}}((u_f^\omega+v)^3) (\tau, \cdot) d\tau.
$$

\subsection{The case $s=1/4$}
In this case,  the numerology is particularly simple
\begin{proposition}\label{prop.fixedbis}
Let us fix $s=1/4$. 
Then there exists $C>0$  such that for every $T\in[0,1]$, every 
$f\in \mathcal{H}^{1/4}(M)$, every $\lambda>0$ for $\omega\in E^{c}_{\lambda,f}$ the map  $K_{f}^{\omega}$ satisfies
\begin{equation}\label{eq.fixedter}
\|K_{f}^{\omega}(v)\|_{X^{1/2}_T} 
\leq C \Big(\lambda^{3}+  \|v\|_{X^{1/2}_T}^{3}\Big),
\end{equation}
\begin{equation}\label{eq.fixedquar}
\|K_{f}^{\omega}(v)-K_{f}^{\omega}(w)\|_{X^{1/2}_T} 
\leq C 
\|v-w\|_{X^{1/2}_T}
\Big(\lambda^{2}+ \|v\|_{X^{1/2}_T}^{2}+\|w\|_{X^{1/2}_T}^{2}\Big)\,.
\end{equation} 
\end{proposition}
\begin{proof}
Indeed, for $\omega \in E^c_{\lambda, f}$, we have 
$$ \|u_f^\omega\|_{L^4((0,T)\times M)}\leq \lambda
$$
and consequently, according to Corollary~\ref{cor.stri},
$$ 
\|K_{f}^{\omega}(v)\|_{X^{1/2}_T} \leq C\|(u_f^\omega+v)^3\|_{L^{4/3}([0,T]\times M)}\leq C \Bigl(\lambda^3 +\|v\|^3_{X^{1/2}_T}\Bigr)
$$ 
and
\begin{multline*}
\|K_{f}^{\omega}(v)- K_f^{\omega}(w)\|_{X^{1/2}_T} \\
\leq C\|(u_f^\omega+v)^3 -(u_f^\omega+w)^3 \|_{L^{4/3}([0,T]\times M)}\leq C\|v-w\|_{X^{1/2}_T} (\lambda^2 +\|v\|^2_{X^{1/2}_T}+\|v\|^2_{X^{1/2}_T})
\end{multline*}
\end{proof}
\subsection{ The case $s>1/4$}
\begin{proposition}\label{prop.fixed}
Let us fix $s>1/4$. 
Then there exists $\sigma>1/2$, $C>0$ and $\kappa>0$ such that for every $T\in[0,1]$, every 
$f\in H^s(M)$, every $\lambda>0$, $\omega\in E^{c}_{\lambda,f}$ the map  $K_{f}^{\omega}$ satisfies
\begin{equation}\label{eq.fixed}
\|K_{f}^{\omega}(v)\|_{X^\sigma_T} 
\leq C \Big(\lambda^{3}+ T^\kappa \|v\|_{X^\sigma_T}^{3}\Big),
\end{equation}
\begin{equation}\label{eq.fixedbis}
\|K_{f}^{\omega}(v)-K_{f}^{\omega}(w)\|_{X^\sigma_T} 
\leq C T^\kappa
\|v-w\|_{X^\sigma_T}
\Big(\lambda^{2}+ \|v\|_{X^\sigma_T}^{2}+\|w\|_{X^\sigma_T}^{2}\Big)\,.
\end{equation} 
\end{proposition}
\begin{proof}
Let us notice that,
according to~Corollary~\ref{cor.stri}, for $1/2<\sigma<1$ (to be fixed later),
$$ 
\|K_{f}^{\omega}(v)\|_{X^\sigma_T} \leq C\|(u_f+v)^3\|_{L^{p'}([0,T]; L^{q'}(M))},
$$ 
where 
$$ \frac 1 p + \frac 3 q = \frac 3 2 - (1- \sigma), \qquad {p=4}.$$
Notice that $4>2/\sigma$ and thus the choice $p=4$ is allowed.
Using the triangle inequality, we obtain
\begin{equation}\label{eq.base}
 \|K_{f}^{\omega}(v)\|_{X^\sigma_T} 
\leq C (\|u_f^{\omega}\|_{L^4([0,T]; L^{3q'}(M))}^3+\|v\|^3_{L^{4}([0,T];L^{3q'}(M))}).
\end{equation}
Let us first study the second term in the right hand-side of (\ref{eq.base}).
This will be done by invoking the $H^\sigma$, $\sigma>1/2$ well-posedness argument.
Observe that $\frac{1}{q'}=\frac{11}{12}-\frac{\sigma}{3}$. Let $p$ be such that $(p,3q')$ is 
$\sigma$-admissible, i.e. 
$$
\frac{1}{p}+\frac{3}{3q'}=\frac 3 2 -\sigma,\quad
\Rightarrow\quad
\frac{1}{p}=\frac{7}{12}-\frac{2\sigma}{3}\,.
$$
Observe that since $\sigma>1/2$ we have that $p>4$. Therefore thanks to the H\"older inequality (in time)
for $\sigma>1/2$ there exists $\kappa>0$ such that
$$
\|v\|^3_{L^{4}([0,T];L^{3q'}(M))}\leq T^{\kappa}\|v\|_{X^\sigma_T}^{3}\,.
$$
Let us next study the first term in the right hand-side of (\ref{eq.base}).
We first consider the case $s\geq 1$. Using the Sobolev inequality, for $\omega\in E^{c}_{\lambda,f}$,
we can write
$$
\|u_f^{\omega}\|_{L^4([0,1]); L^{3q'}(M))}\leq 
C\|u_f^{\omega}\|_{L^4([0,1]; W^{s-\frac 1 4, 4} (M))}
\leq C \lambda\,.
$$
This ends the proof of~\eqref{eq.fixed} for $s\geq 1$ ($\sigma$ being an arbitrary number in $(1/2,1)$).

Let us next assume that $s<1$. Then for $\omega\in E^{c}_{\lambda,f}$ 
(and according to Sobolev embedding), we have 
$$ 
\|u_f^{\omega}\|_{L^4([0,1]); L^{q_0}(M))}\leq 
C\|u_f^{\omega}\|_{L^4([0,1]; W^{s-\frac 1 4, 4} (M))}
\leq C \lambda
$$
where  
$$ \frac 1 {q_0} = \frac 1 4 - \frac {(s- 1/ 4)} 3\,.
$$
We choose
$$ 
\sigma = \min\big(\frac{9}{10}, \frac{1}{2} + 3(s- \frac 1 4)\big)$$
which fixes the value of $\sigma$ in the case $s<1$.
Then
\begin{equation*}
\frac 1 4 + \frac 3 q = \frac 3 2 - (1- \sigma) \Rightarrow \frac 1 q = 
\frac 1 4 - \frac{( 1/ 2 - \sigma)} 3\Rightarrow \frac 1 {3 q'}\geq \frac 1 4 - 
\frac{(s- 1 /4)} 3= \frac 1 {q_0}\,.
\end{equation*}
As a consequence for $\omega\in E^{c}_{\lambda,f}$, 
using the H\"older inequality in space, we get that for $s<1$, we can 
bound the contribution of the first term in the right hand-side of~\eqref{eq.base} 
as follows
$$
\|u_f^{\omega}\|_{L^4([0,1]); L^{3q'}(M))}\leq \|u_f^{\omega}\|_{L^4([0,1]); L^{q_0}(M))}\leq 
C\|u_f^{\omega}\|_{L^4([0,1]; W^{s-\frac 1 4, 4} (M))}
\leq C \lambda.
$$ 
This ends the proof of~\eqref{eq.fixed}. 
The proof of~\eqref{eq.fixedbis} is similar. Indeed, it suffices to write
\begin{multline*}
\|K_{f}^{\omega}(v)-K_{f}^{\omega}(w)\|_{X^\sigma_T} \leq
\|v-w\|_{L^{4}([0,T]; L^{3q'}(M))}\Big(\|u_{f}^{\omega}\|^2_{L^{4}([0,T]; L^{3q'}(M))}
\\
+\|v\|^2_{L^{4}([0,T]; L^{3q'}(M))}+\|w\|^2_{L^{4}([0,T]; L^{3q'}(M))}\Big)
\end{multline*}
and to use the previous estimates. This completes the proof of Proposition~\ref{prop.fixed}.
\end{proof}

Let us now complete the proof of Theorem~\ref{thm1}.
Let us first consider the case of a randomization induced by a general family of random variables 
satisfying (\ref{norm}). Fix $0<T\leq 1$.
As a consequence of Propositions~\ref{prop.fixedbis} and \ref{prop.fixed}, if $\omega \in E^c_{\lambda,T,f}$ and if 
\begin{equation}\label{eq.cond}
C \lambda^3 + T^\kappa ( 2C \lambda^3)^3\leq  2C \lambda^3,\text{ and } C T^\kappa ( \lambda^2 + \lambda^6)\leq \frac 1 2,  \quad \kappa \geq 0,\, 
(\kappa >0 \text{ if $s>1/4$}),
\end{equation}
then  the map $K_{f}^{\omega}$ is a contraction on the ball of radius $2C\lambda^{3}$ of 
$X^{\sigma}_{T}$. Notice that the condition~\eqref{eq.cond} above is implied
by the following 
$$ T^\kappa \lambda^6 = \epsilon ^6 \ll 1.$$ 
As a consequence, if we define, with $\delta = \kappa /6$, 
$$ \Omega_T = E^c_{\lambda=\epsilon T^{- \delta},T,f}, 
\quad  
\Sigma = \bigcup_{n\in \mathbb{N}^*} \Omega_{1/n},
$$
then
$$
p(\Omega_T) \geq 1 - CT^{1+4\delta}, \quad p(\Sigma )=1
$$
and we obtain the first part in Theorem~\ref{thm1} in the case of general
variables satisfying only (\ref{norm}).
 
Let us finally consider the case of random variables satisfying in addition to (\ref{norm}) the property
(\ref{property}).
In this case $\lambda=\lambda(T)$ is chosen such that
$$ T^\kappa \lambda^6= \epsilon \ll 1$$
and according to Proposition~\ref{prop.fixed}, if $\omega \in
E^c_{\lambda,T,f}$, then the map $K_{f}^{\omega}$ 
is a contraction on the ball of radius $2C\lambda^{3}$ of 
$X^{\sigma}_{T}$. Now according to~\eqref{iyud}, we obtain that if we set
$$ 
\Omega_T = E^c_{\lambda(T),f}, \quad  \Sigma = \bigcup_{n\in \mathbb{N}^*}
\Omega_{1/n},
$$
then
$$
p(\Omega_T)\geq 1 - Ce^{-c/T^\delta},\,\, \delta, C,c>0, \quad p(\Sigma )=1.
$$
This completes the proof of Theorem~\ref{thm1}.
\qed
\begin{remarque}
{\rm
Let us observe that in Theorem~\ref{thm1}, if $s>1/4$ then $\sigma>1/2$.
}
\end{remarque}
\section{Manifolds with boundary}\label{sec.bord}
In this section we consider the case of Dirichlet or Neumann boundary conditions. For conciseness, we shall drop the subscript $D,N$.
\subsection{Strichartz and spectral projector estimates}
The following spectral projector estimate is proved by Smith and Sogge~\cite{SmSo06}
\begin{proposition}\label{chrisbis}
There exists $C>0$ such that for every $n\geq 1$,
$$
\|e_n\|_{L^{5}(M)}\leq C(1+\lambda_{n}^2)^{\frac{1}{5}}\,.
$$
\end{proposition}
This  estimate implies (see~\cite[Theorem 2]{BuLePl06}) the following Strichartz inequality
\begin{proposition}\label{stribord}
There exists $C>0$ such that for 
$$ \|e^{\pm it \sqrt{- {\mathbf \Delta}} }f\|_{L^5((0,1)\times M)} \leq C \|f\|_{H^{\frac 7 {10}}(M)}\,.
$$
\end{proposition} 
By interpolation and duality, we deduce that the Strichartz inequalities~\eqref{eq.homogbis} and~\eqref{eq.inhomog} remain true provided we replace the definition of admissible couples by 
\begin{definition}
Let $0\leq s\leq 1$. A couple $(p,q)$ is $s$-admissible if
$$
\frac 1 p + \frac 3 q = \frac 3 2 -s
$$
and
\begin{equation*}
p \geq 
\begin{cases} 
&\frac 7 {2s} \text{ if } s \leq \frac 7 {10}\,,\\
&5  \text{ if } s \geq \frac 7 {10}\,.
\end{cases}
\end{equation*}
\end{definition}
\subsection{Averaging effects}
\begin{proposition}\label{prop.averbis}
Let $s\geq\frac 2 5$.
Under the assumptions of Theorem~\ref{thm2}, 
for $f=(f_1,f_2)\in {\mathcal H}^{s}(M)$, we consider the free evolution with data $f^{\omega}$, given by 
$$
u_{f}^\omega(t,x) = 
\cos(t\sqrt{- {\mathbf \Delta}}) f_1^\omega+\frac{\sin ( t \sqrt{- {\mathbf \Delta}})f_2^\omega}{ \sqrt{ - {\mathbf \Delta}}}\,.
$$
Then there exists $C>0$ such that for every $f\in {\mathcal H}^{s}(M)$,
\begin{equation}\label{averbis}
\|(-{\mathbf \Delta} +1)^{\frac s 2 - \frac 1 5}u_f^\omega\|_{L^5(\Omega_\omega\times [0,T]_t\times M_x)} 
\leq C T^{1/5}\|f\|_{{\mathcal H}^{s}(M)}\,.
\end{equation}
In particular, thanks to the Bienaym\'e-Tchebichev inequality, if we set
$$
E_{\lambda,T,f}= \Big\{ \omega \in \Omega\,:\,
\|(-{\mathbf \Delta} +1)^{\frac s 2 - \frac 1 5}u_f^\omega\|_{L^5([0,T]_t\times M_x)} \geq \lambda\Big\}
$$
then there exists $C>0$ such that for every $\lambda>0$, every $f\in {\mathcal H}^{s}(M)$,
$$ 
p(E_{\lambda,T,f})\leq CT \lambda ^{-5}\,
\|f\|^5_{{\mathcal H}^{s}(M)}\,.
$$
\end{proposition}
\begin{proof}
Using Lemma~\ref{lem.k} we compute
\begin{multline*}\|(-{\mathbf \Delta} +1)^{\frac s 2 - \frac 1 5}e^{it \sqrt{ - {\mathbf \Delta}}}f_1^\omega\|_{L^5(\Omega_\omega\times [0,T]_t\times M_x)}\\
\begin{aligned}
&= \| \sum_{n}e^{it \lambda_n}(1+ \lambda^2_n)^{\frac s 2 - \frac 1 5} \alpha_n e_n(x)h_n( \omega)\|_{L^5(\Omega_\omega\times [0,T]_t\times M_x) }\\
&\leq  C\| \Bigl(\sum_{n}|e^{it \lambda_n}(1+ \lambda^2_n)^{\frac s 2 - \frac 1 5} \alpha_n e_n(x)|^2 \Bigr)^{1/2}\|_{L^5([0,T]_t\times M_x) }\\
&\leq  C\| \sum_{n}|e^{it \lambda^2_n}(1+ \lambda^2_n)^{\frac s 2 - \frac 1 5} \alpha_n e_n(x)|^2 \|^{1/2}_{L^{5/2}([0,T]_t\times M_x) }\\
&\leq C\Bigl(\sum_{n} \| |(1+ \lambda^2_n)^{\frac s 2 - \frac 1 5} \alpha_n e_n(x)|^2 \|_{L^{5/2} ( [0,T]_t\times M_x) }\Bigr)^{1/2}
\end{aligned}
\end{multline*}
Finally, using Proposition~\ref{chrisbis}, we obtain
\begin{multline*}
\|(-{\mathbf \Delta} +1)^{\frac s 2 - \frac 1 5}e^{it \sqrt{ - {\mathbf \Delta}}}f_1^\omega\|_{L^5(\Omega_\omega\times [0,T]_t\times M_x)}\\
\leq CT^{1/5}\Bigl(\sum_{n} | (1+ \lambda^2_n)^{s} \alpha_n |^2 \Bigr)^{1/2}\leq C  T^{1/5}\|f_1\|_{H^s(M)}\,.
\end{multline*}
The contribution of $f_2$ is dealt with similarly. This ends the proof of Proposition~\ref{prop.averbis}.
\end{proof}
\subsection{The fixed point}
In this section we shall prove only the case $s= 8/21$ in Theorem~\ref{thm2}. The case $s>8/21$ and the improved estimate for Gaussian are proved {\em mutatis mutandi} following the strategy developed in Section~\ref{sec.fix}.
Interpolating between~\eqref{averbis} and the trivial bound
\begin{equation*}
\|u_f^\omega\|_{L^2(\Omega_\omega\times [0,T]_t\times M_x)} 
\leq C T^{1/2}\|f\|_{\mathcal{H}^0(M)}\,.
\end{equation*}
gives
\begin{equation}\label{averter}
\|u_f^\omega\|_{L^{14/3}(\Omega_\omega\times [0,T]_t\times M_x)} 
\leq C T^{3/14}\|f\|_{\mathcal{H}^{8/21}(M)}\,.
\end{equation}
As in Section~\ref{sec.fix}, we are looking for a fixed point of the map
$$
K_{f}^{\omega} \,:\,v
\longmapsto 
-\int_0^t 
\frac{\sin((t-\tau) \sqrt{ - {\mathbf \Delta}})} { \sqrt{ - {\mathbf \Delta}}}((u_f+v)^3) (\tau, \cdot) d\tau.$$
Using the Strichartz inequalities in Section~\ref{sec.stri} (with the new definition of admissible couples and consequently of $X^s_T$ and $Y^s_T$ spaces), we obtain 
$$ 
\|K_{f}^{\omega}\|_{X^{2/3}_T} \leq C \|(u_f+v)^3\|_{Y^{1/3}_T}\,.
$$
But (observe that $(21/4, 14/3)$ is a $2/3$-admissible couple)
$$ \|g\|_{L^\infty((0,T); H^{2/3}(M))} + \|g\|_{L^{21/4}((0,T); L^{14/3}(M))} \leq C\|g\|_{X^{2/3}_T}
$$
and (observe that $(21/2, 14/5)$ is a $1/3$-admissible couple)
$$ \|v^3\|_{Y^{1/3}_T} \leq C \|v^3\|_{L^{21/19}((0,T); L^{14/9}(M))}\leq C\|v\|^3_{L^{63/19}((0,T); L^{14/3}(M))}\leq C_T \|v\|^3_{X^{2/3}_T}\,.
$$
These a priori bounds combined with the estimate~\eqref{averter} (and the fact that $21/4 >63/19$) allow to perform the fixed point in the space $X^{2/3}_T$ (for sufficiently small $T$ depending on $\omega$), exactly as in the previous section. This ends the proof of Theorem~\ref{thm2}.

\appendix

\section{Ill posedness on $\mathcal{H}^s(M)$, $s<1/2$.}
The goal of this appendix is to prove the ill-posedness statement claimed in the introduction. For that purpose, we first prove the lack of continuity at $0$ of the flow map on $\mathcal{H}^s(M)$, $s<1/2$. More precisely we have the following result.
\begin{proposition}\label{ill}
Let us fix $s\in ]0,1/2[$. Then there exists $\delta >0$ and a sequence  $(t_n)_{n=1}^{\infty}$ 
of positive numbers tending to zero and a sequence $(u_n(t))_{n=1}^{\infty}$ of $C^{\infty}(M)$ functions
such that
$$
(\partial_{t}^{2}-{\mathbf \Delta})u_n+u_n^3=0
$$
with
$$
\|u_{n}(0)\|_{{\mathcal H}^{s}(M)}\leq C \log(n)^{-\delta}\rightarrow_{n\rightarrow + \infty} 0
$$
but
$$
\|u_{n}(t_n)\|_{{\mathcal H}^{s}(M)}\geq C \log(n)^{\delta}\rightarrow_{n\rightarrow + \infty} +\infty.
$$
\end{proposition}

\begin{proof}
The proof is strongly inspired by the considerations in \cite{ChCoTa03} where the Euclidean space is 
considered instead of a riemannian manifold $(M,g)$.
The only advantage of our argument with respect to
\cite{ChCoTa03} is that we avoid the scaling consideration of \cite{ChCoTa03} and thus we can keep 
the argument local in space and thus it can still work for (\ref{1}) posed on a compact manifold.
A similar discussion in the context of NLS may be found in \cite{BGT}. 

Working in a local coordinate system near a fixed point of $M$, we consider an initial data concentrating
at this fixed point. Namely, we consider (\ref{1}) subject to initial conditions
\begin{equation}\label{data}
(f_{1,n}(x),f_{2,n}(x))=(\kappa_{n}n^{\frac{3}{2}-s}\varphi(nx), 0),\qquad n\gg 1\,,
\end{equation}
where $\varphi$ is a nontrivial bump function on $\R^3$ and
$$
\kappa_{n}\equiv [\log(n)]^{-\delta_1},
$$
with $\delta_1>0$ to be fixed later.
The equation (\ref{1}) being $H^1(M)$ sub critical (and defocusing), we obtain that 
(\ref{1}) with data given by (\ref{data}) has a unique global smooth solution which we denote by $u_{n}$.
We will consider the solution of (\ref{1}) with data (\ref{1}) only for small times and thanks to the finite
propagation speed of the wave equation the analysis is local.
Next, let us denote by $V(t)$ the global solution of the ODE
\begin{equation}\label{V}
V''+V^3=0,\quad V(0)=1,\,\, V'(0)=0.
\end{equation}
Multiplying (\ref{V}) by $V'$, we deduce that $V(t)$ is a periodic function.
We next denote by $v_n$ the solution of
\begin{equation}\label{ode}
\partial_{t}^2v_{n}+v_{n}^{3}=0,\quad (v_{n}(0),\partial_{t}v_{n}(0))=
(\kappa_{n}n^{\frac{3}{2}-s}\varphi(nx), 0).
\end{equation}
It is now clear that
$$
v_{n}(t,x)=\kappa_{n}n^{\frac{3}{2}-s}\varphi(nx)V\Big(t\kappa_{n}n^{\frac{3}{2}-s}\varphi(nx)\Big).
$$  
We next consider the semi-classical energy
$$
E_{n}(u)\equiv n^{-(1-s)}\big(\|\partial_{t}u\|_{L^{2}(M)}^{2}+
\|\nabla u\|_{L^{2}(M)}^{2}\big)^{\frac{1}{2}}+n^{-(2-s)}\big(\|\partial_{t}u\|_{H^{1}(M)}^{2}+\|\nabla
u\|_{H^{1}(M)}^{2}\big)^{\frac{1}{2}}\,.
$$
We are going to show that for very small times $u_n$ and $v_n$ are close with respect to $E_n$
but these small times are long enough to get the needed amplification of the $H^s$ norm (this
amplification is a phenomenon only related to the solution of (\ref{ode})).
Here is the precise statement.
\begin{lemme}\label{compar}
There exist $\varepsilon>0$, $\delta_2>0$ and $C>0$ such that, if we set
$$
t_{n}\equiv [\log(n)]^{\delta_2}n^{-(\frac{3}{2}-s)}
$$
then for every $n\gg 1$, every $t\in [0,t_n]$,
$
E_{n}(u_{n}(t)-v_{n}(t))\leq Cn^{-\varepsilon}\,.
$
Moreover, 
\begin{equation}\label{s}
\|u_{n}(t)-v_n(t)\|_{H^{s}(M)}\leq Cn^{-\varepsilon}\,.
\end{equation}
\end{lemme}
\begin{proof}
Set $w_{n}=u_{n}-v_{n}$. Then $w_n$ solves the equation
\begin{equation*}
(\partial_{t}^{2}-{\mathbf \Delta})w_n={\mathbf \Delta} v_n-3v_{n}^{2}w_{n}-3v_{n}w_{n}^{2}-w_{n}^{3},\quad
(w_{n}(0,\cdot),\partial_{t}w_{n}(0,\cdot))=(0,0)\,.
\end{equation*}
Set
$$
F\equiv{\mathbf \Delta} v_n-3v_{n}^{2}w_{n}-3v_{n}w_{n}^{2}-w_{n}^{3}\,.
$$
By the energy inequality inequality for the wave equation, we get
\begin{equation*}
\frac{d}{dt}\Big(E_{n}(w_n(t))\Big)
\leq C n^{-(2-s)}\|F(t,\cdot)\|_{H^1(M)}+Cn^{-(1-s)}\|F(t,\cdot)\|_{L^2(M)}\,.
\end{equation*}
We have for $t\in[0,t_n]$,
\begin{equation*}
\|{\mathbf \Delta}(v_n)(t,\cdot)\|_{H^1(M)}\leq C[\log(n)]^{3\delta_2}n^{3-s},\quad
\|{\mathbf \Delta}(v_n)(t,\cdot)\|_{L^2(M)}\leq C[\log(n)]^{2\delta_2}n^{2-s}\,.
\end{equation*}
Therefore
\begin{equation}\label{energia}
\frac{d}{dt}\Big(E_{n}(w_n(t))\Big)
\leq C\Big([\log(n)]^{3\delta_2}n+n^{-(2-s)}\|G(t,\cdot)\|_{H^1(M)}+n^{-(1-s)}\|G(t,\cdot)\|_{L^2(M)}\Big)\,,
\end{equation}
where
$$
G\equiv -3v_{n}^{2}w_{n}-3v_{n}w_{n}^{2}-w_{n}^{3}\,.
$$
Writing for $t\in [0,t_n]$,
$$ w_n(t,x) = \int_0^t \partial_s w_n(s, x) ds,
$$
we obtain
\begin{equation}\label{i}
\|w_{n}(t,\cdot)\|_{H^k(M)}\leq C[\log(n)]^{\delta_2}n^{-(\frac{3}{2}-s)}
\sup_{0\leq \tau\leq t}\|\partial_{t}w_{n}(\tau,\cdot)\|_{H^{k}(M)}\,.
\end{equation}
Moreover, we have that for $t\in [0,t_n]$,
\begin{equation}\label{ii}
\|\nabla^{k}v_{n}(t,\cdot)\|_{L^{\infty}(M)}\leq C[\log(n)]^{k\delta_2}n^{\frac{3}{2}-s+k}
\end{equation}
and thanks to the Gagliardo-Nirenberg inequality,
\begin{equation}\label{iii}
\|w_{n}(t,\cdot)\|_{L^{\infty}(M)}\leq C\|w_{n}(t,\cdot)\|_{H^2(M)}^{\frac{3}{4}}\|w_n(t,\cdot)\|_{L^2(M)}^{\frac{1}{4}}
\leq Cn^{\frac{3}{2}-s}E_{n}(w_{n}(t))\,.
\end{equation}
Set
$$
e_{n}(w_n(t))\equiv \sup_{0\leq \tau\leq t}E_{n}(w_n(\tau))\,.
$$
Using (\ref{i}), (\ref{ii}) and (\ref{iii}), we get that for $l=1,2$,
$$
n^{-(l-s)}\|G(t,\cdot)\|_{H^{l-1}(M)}\leq C[\log(n)]^{l\delta_2}n^{\frac{3}{2}-s}\big(e_{n}(w_n)+[e_{n}(w_n)]^3\big)\,.
$$
Therefore, coming back to (\ref{energia}), we get
\begin{equation*}
\frac{d}{dt}\Big(E_{n}(w_n(t))\Big)\leq 
C[\log(n)]^{3\delta_2}n+C[\log(n)]^{l\delta_2}n^{\frac{3}{2}-s}\big(e_{n}(w_n)+[e_{n}(w_n)]^3\big)\,.
\end{equation*}
We first suppose that $e_{n}(w_n(t))\leq 1$ which holds for small values of $t$
since $w_{n}(0)=0$. 
Thanks to a Gronwall lemma argument for $t\in [0,t_n]$,
$$
e_{n}(w_{n}(t))
\leq 
C[\log(n)]^{\delta_2}n^{s-\frac{1}{2}}e^{Ct[\log(n)]^{2\delta_2}n^{\frac{3}{2}-s}}
\leq 
C[\log(n)]^{\delta_2}n^{s-\frac{1}{2}}e^{C[\log(n)]^{2\delta_2}}\,.
$$
(one should see $\delta_2$ as $3\delta_2-2\delta_2$ and $s-1/2$ as $1-(3/2-s)$).
Since $s<1/2$, by taking $\delta_2>0$ small enough, we obtain that there exists $\varepsilon>0$ such that
$$
E_{n}(w_n(t))\leq Cn^{-\varepsilon}
$$
and in particular one has for $t\in [0,t_n]$,
\begin{equation}\label{parvo}
\|\partial_{t}w_{n}(t,\cdot)\|_{L^2(M)}+\|w_{n}(t,\cdot)\|_{H^1(M)}\leq
Cn^{1-s-\varepsilon}\,.
\end{equation}
We next estimate $\|w_{n}(t,\cdot)\|_{L^2}$. 
We may write for $t\in[0,t_n]$,
$$
\|w_{n}(t,\cdot)\|_{L^2(M)}= \|\int_0^t \partial_tw_{n}(s,\cdot)\|_{L^2(M)}\leq ct_{n}\sup_{0\leq\tau\leq t}\|\partial_{t}w_{n}(\tau,\cdot)\|_{L^2(M)}\,.
$$
Thanks to (\ref{parvo}) and the definition of $t_n$, we get
$$
\|w_{n}(t,\cdot)\|_{L^2(M)}\leq
C[\log(n)]^{\delta_2}n^{-(\frac{3}{2}-s)}n^{1-s}n^{-\varepsilon}\,.
$$
Therefore, since $s<1/2$, 
\begin{equation}\label{vtoro}
\|w_{n}(t,\cdot)\|_{L^2(M)}\leq Cn^{-s-\varepsilon}\,.
\end{equation}
An interpolation between (\ref{parvo}) and (\ref{vtoro}) yields (\ref{s}). 
This completes the proof of Lemma~\ref{compar}.
\end{proof}
Using Lemma~\ref{compar}, we may write
$$
\|u_{n}(t_n,\cdot)\|_{H^{s}(M)}\geq \|v_{n}(t_n,\cdot)\|_{H^{s}(M)}-Cn^{-\varepsilon}\,.
$$
On the other hand from the representation of $v_n$, we obtain for $n$ large enough
\begin{equation}\label{convex}
\|v_{n}(t_n,\cdot)\|_{H^{s}(M)}\geq C\kappa_{n}(t_n\kappa_{n}n^{\frac{3}{2}-s})^{s}
=C[\log(n)]^{-(s+1)\delta_1+s\delta_2}\,.
\end{equation}
Indeed this estimate is the consequence of the following lemma.
\begin{lemme}\label{lem.expl}
Consider a smooth non constant $2 \pi$ periodic function $V$ and  two functions $\psi,\phi \in C^\infty_0( \mathbb{R}^d)$ such that $\phi \psi $ is not identically vanishing. Then there exists $C>0$ such that for any $\lambda >1$ and any $s\geq 0$
$$ 
\| \psi (x) V(\lambda \phi(x))\|_{H^s(\mathbb{R}^d)} \geq \frac {\lambda^s} C -C\,.
$$
\end{lemme}
\begin{proof} The multiplication by a smooth function being continuous on $H^s$, it suffices to prove the estimate with $\psi$ replaced by any function $\chi \times \psi$ with $\chi \in C^\infty _0$. As a consequence (and using that $H^s$ is invariant by diffeomorphisms), we can assume that on the support of the function $\psi$, we have $\phi(x) = x_1$. We develop
$$ V(t) = \sum_{n\in \mathbb{Z}}v_n e^{int},\quad |v_n| \leq C_N (1+|n|)^{-N}$$
and replacing the function $V$ by $V-v_0$ (which changes the $H^s$ norm by at most a constant), we can assume $v_0=0$. Choose $n_1\neq 0$ such that $v_{n_1}\neq 0$ ($V$ is non constant). Then
\begin{equation*}\begin{aligned} \| \psi (x) V(\lambda x_1))\|_{H^s(\mathbb{R}^d)}^2
&= \int \Bigl| \sum_n v_n \widehat{\psi} ( \xi_1 - n \lambda, \xi') \Bigr|^2 ( 1+ |\xi|)^{2s} d \xi\\
&\geq\int_{\smash{|(\xi_1- n_1 \lambda, \xi')|\leq 1}} \bigl| \sum_n v_n \widehat{\psi} ( \xi_1 - n \lambda, \xi') \bigr|^2 ( 1+ |\xi|)^{2s} d \xi \\
& \geq  \frac 1 2\int_{\smash{|(\xi_1- n_1 \lambda, \xi')|\leq 1}}\bigl| v_{n_1} \widehat{\psi} ( \xi_1 - n_1 \lambda, \xi') \bigr|^2 ( 1+ |\xi|)^{2s} \\
&\qquad \hfill -2\Bigl|\sum_{n\neq n_1} v_{n} \widehat{\psi} ( \xi_1 - n \lambda, \xi') \Bigr|^2 ( 1+ |\xi|)^{2s} d \xi\,.
\end{aligned}
\end{equation*}
The first term in the right hand side is bounded from below by 
\begin{equation*}
\frac 1 2 \int_{|\xi|\leq 1} \bigl| v_{n_1} \widehat{\psi} ( \xi_1, \xi') \bigr|^2 ( 1+ |\xi_1+n_1 \lambda|)^{2s}
\geq    \frac 1 C |n_1 \lambda|^{2s}- C 
\end{equation*}
whereas the second term is bounded (in absolute value) by 
\begin{multline*}
2 \int_{|\xi|\leq 1}\Bigl|\sum_{n \neq n_1} v_{n}\widehat{\psi} ( \xi_1 +(n_1- n) \lambda, \xi') \Bigr|^2 ( 1+ |\xi|+ |n_1 \lambda|)^{2s} d \xi\\\leq  C \Bigl(\sum_{n\neq n_1} |v_n||(n_1- n) \lambda -1|^{-N} |n_1 \lambda|^{s}\Bigr)^2
\leq C 
\end{multline*}
which ends the proof of Lemma~\ref{lem.expl}.
\end{proof}
By choosing $\delta_1$ small enough (depending on $\delta_2$ fixed in Lemma~\ref{compar}), 
we obtain that
$\lim_{n\rightarrow \infty}\|v_n(t_n,\cdot)\|_{H^s}=\infty$
which implies that
$\lim_{n\rightarrow \infty}\|u_n(t_n,\cdot)\|_{H^s}=\infty$.
This completes the proof of Proposition~\ref{ill}.
\end{proof}
We now can show that these solutions we just constructed can be glued together to give the following statement.
\begin{proposition}\label{prop.ill}
Let us fix $s\in ]0,1/2[$. There exists an initial data $f= (f_1,f_2) \in \mathcal{H}^s(M)$ such there exists no solution of~\eqref{1} in $L^ \infty((-T,T); \mathcal{H}^s(M))$, $T>0$  
with initial data $f$ satisfying in addition the finite speed of propagation.
\end{proposition}
\begin{proof} 
We use the notations introduced in the proof of the previous proposition and consider solutions 
$u_n$ as constructed above, but, working in local coordinates, with initial data centered at points 
$x_n= (x_{n,1}, x_n'=0)$ with a sequence $x_{n,1}$ converging to $0$ to be specified later. 
As a consequence, the support of the initial data of $u_n$ is included in the set  
$$\{x= (x_1, x')\in \mathbb{R}^3\,:\, |x_1- x_{n,1}| +|x'| \leq \frac C n\}.$$
Furthermore, the explicit form of $v_n$ (and the fact that $t_n \ll n^{-1}$) shows that if 
$$ K_n=\{x\,:\, |x_1-x_{n,1}| \leq \frac {C} {2n} -Ct_n\}$$ then 
$$ \|u_n(t_n, \cdot)\|_{H^{s} (K_n)}\rightarrow + \infty.$$
Remark also that we have 
\begin{lemme}\label{restric}
For any $0\leq s <1/2$, there exists $C>0$ such that 
$$ \forall\, n, \,\, \forall\, u \in H^s(M), \quad   \|u \|_{H^s(M)} \geq C \|u\mid_{K_n} \|_{H^s(K_n)}\,.
$$
\end{lemme}
\begin{proof}
Indeed, the multiplication by the Heaviside function is continuous on $H^s(M)$ (because $-1/2<s<1/2$ and
$$ 
u\mid_{K_n}= u \times 1_{-\frac {C} {2 n}+ Ct_n<x_1-x_{n,1}<\frac {C} {2 n} -Ct_n}\,.
$$
\end{proof}
We now consider a sequence $(n_k)_{k\in \mathbb{N}}$
such that
\begin{itemize}
\item $ \|u_{n_k}(0,\cdot)\|_{H^s(M)}\leq 2^{-k}$,
\item $n_k\leq 2^{-k}$.
\end{itemize}
Remark also that all the estimates on the functions $u_{n_k}$ are  independent of the choice of the 
sequence $x_{n,1}$ (because the bounds on $u_n$ we have are independent of the
choice of the concentration point), 
and consequently, we can assume that $x_{n_k,1}=\frac 1 {k^2}$. Consider now as initial data 
$$ 
f= (f_1=\sum_{k\geq k_0} u_{n_k}(0), f_2=0),
$$
where $k_0\geq 1$ is a large constant.
The support of the function $f_1$ is included in the union of the balls of radius $C2^{-k}$ 
centered at $(\frac 1 {k^2},0,0)$. 
If $k_0$ is large enough,
any solution $u$ of the non linear wave equation with initial data $f$, 
satisfying the finite speed of propagation will consequently coincide with the solutions 
$u_{n_k}$, $k\geq k_0$ we just constructed on the cone $ K_{n_k}$ (notice that for $k_0\gg 1$, the cones 
$K_{n_k}$, $k\geq k_0$ are disjoint).
As a consequence, these solutions will satisfy (using Lemma~\ref{restric})
$$ \|u(t_{n_k},\cdot)\|_{\mathcal{H}^s(M)} \geq C\|u(t_{n_k},\cdot)\|_{\mathcal{H}^s(K_{n_k})}
\rightarrow_{k\rightarrow + \infty} + \infty
$$
and consequently
$$ 
\limsup_{t\rightarrow 0^+}\|u(t,\cdot)\|_{\mathcal{H}^s(M)}= + \infty \,.
$$
This ends the proof of Proposition~\ref{prop.ill}
\end{proof}
\section{Lack of $H^s$ regularization under the considered randomization}
The goal of this appendix is to give the proof of the following lemma.
\begin{lemme}\label{last}
Let 
$$
f=\sum_{n=1}^{\infty}\alpha_n e_{n}(x)\in H^{s}(M)
$$ 
be such that for some $\varepsilon>0$ one has that $f$ does not belong to 
$H^{s+\varepsilon}(M)$. Let $(l_n(\omega))_{n=1}^{\infty}$ be a sequence of independent random 
variables such that there exists $c>0$ satisfying
$$\limsup_{n\rightarrow + \infty}p(\{|l_n| \leq c\})<1,$$
(notice that this assumption is fulfilled if the random variables are
identically distributed and non identically zero). 
If we set
$$
f^{\omega}=\sum_{n=1}^{\infty}l_{n}(\omega)\alpha_n e_{n}(x),
$$
then we have that $f^{\omega}$ belongs to $H^{s+\varepsilon}(M)$ with probability zero.
\end{lemme}
\begin{proof}
A similar argument is given in \cite{BuTz07}.
Denote by $\mu_n$ the distribution of $(l_n(\omega))_{n=1}^{\infty}$.
By assumption there exists $c, \delta>0$ such that $\mu_n([-c,c])\leq (1- \delta)$. 
Then, we can write (with $\rho_n=e^{-c^2\lambda_{n}^{2(s+\varepsilon)}\alpha_{n}^{2}}$) 
\begin{multline*}
\int_{\Omega}
e^{-\|f^{\omega}\|^{2}_{H^{s+\varepsilon}(M)}}dp(\omega)=
\prod_{n=1}^{\infty}
\int e^{-\lambda_{n}^{2(s+\varepsilon)}\alpha_{n}^{2}x^2}d\mu_n(x)\\
= \int_{-c}^{c}e^{-\lambda_{n}^{2(s+\varepsilon)}\alpha_{n}^{2}x^2}d\mu_n(x)+\int_{|x|\geq c}e^{-\lambda_{n}^{2(s+\varepsilon)}\alpha_{n}^{2}x^2}d\mu_n(x)\\
\leq
\prod_{n=1}^{\infty}\Big(\mu_n(-c,c)+\rho_n(1-\mu_n(-c,c))\Big)
=
\prod_{n=1}^{\infty}\Big(\mu_n(-c,c)(1-\rho_n)+ \rho_n\Big)\\
\leq \prod_{n=1}^{\infty}\Big((1- \delta)(1-\rho_n)+ \rho_n\Big)
\leq \prod_{n=1}^{\infty}\Big(1- \delta(1-\rho_n)\Big).
\end{multline*}
Since by assumption $\sum_{n}\lambda_{n}^{2(s+\varepsilon)}\alpha_{n}^{2}=\infty$, we obtain that
$
\sum_{n=1}^{\infty}(1-e^{-c^2\lambda_{n}^{2(s+\varepsilon)}\alpha_{n}^{2}})=\infty
$
and therefore
$$
\prod_{n=1}^{\infty}\Big(1- \delta(1-e^{-c^2\lambda_{n}^{2(s+\varepsilon)}\alpha_{n}^{2}})=0\quad \Rightarrow \quad  \int_{\Omega}e^{-\|f^{\omega}\|^{2}_{H^{s+ \varepsilon}(M)}}dp(\omega)=0\,.
$$
This implies that $\|f^{\omega}\|^{2}_{H^{s+\varepsilon}(M)}=\infty$ almost surely.
This completes the proof of Lemma~\ref{last}.
\end{proof}

\end{document}